\documentclass[11 pt,reqno]{amsart}

\usepackage[utf8]{inputenc}
\usepackage{amsmath}
\usepackage{amssymb}
\usepackage{amsthm}
\usepackage{amstext} 

\usepackage[hidelinks]{hyperref}
\usepackage{color,soul}

\usepackage{nicematrix}
\usepackage[shortlabels]{enumitem}
\usepackage{booktabs}
\newcommand{\ra}[1]{\renewcommand{\arraystretch}{#1}}
\usepackage{makecell}

\usepackage{geometry}
\setul{}{0.3 ex}

\usepackage[title]{appendix}

\newcommand{\R}{\mathbb{R}}

\newcommand{\C}{\mathbb{C}}

\newcommand{\id}{\operatorname{Id}}
\newcommand{\li}{\langle}
\newcommand{\ri}{\rangle}

\newcommand{\Ss}{\mathbb{S}}

\newcommand{\SO}{\mathrm{SO}}

\newcommand{\U}{\mathrm U}
\newcommand{\SU}{\mathrm{SU}}

\newcommand{\Sp}{\mathrm{Sp}}

\newcommand{\cycl}[1]{\underset{#1}{\mathfrak{S}}}

\newcommand{\contr}{\,\lrcorner\,}

\makeatletter
\newcommand\columntag[2]{#1\def\@currentlabel{#1}.\label{#2}}
\makeatother

\newtheorem{theorem}{Theorem}
\newtheorem{theoremalpha}{Theorem} 

\newtheorem{lemma}[theorem]{Lemma}
\newtheorem{proposition}[theorem]{Proposition}

\numberwithin{equation}{section}
\numberwithin{table}{section}
\numberwithin{theorem}{section}

\theoremstyle{definition}
\newtheorem{definition}[theorem]{Definition}
\newtheorem{remark}[theorem]{Remark}

\title[Hypersurfaces of 6-dim nearly Kähler manifolds]{Hypersurfaces of six-dimensional nearly Kähler manifolds}
\date{}
\author{Mateo Anarella}
\author{Marie D'haene}
\address{ KU Leuven, Department of Mathematics, Celestijnenlaan 200 B – Box 2400, 3001 Leuven, Belgium}
\email{mateo.anarella@kuleuven.be}
\email{marie.dhaene@kuleuven.be}

\thanks{M.\ Anarella is partially supported by FWO and FNRS under EOS project G0I2222N; M.\ D'haene is supported by Methusalem grant METH/21/03–long term structural funding of the Flemish Government.}


\begin{document} 
\begin{abstract}
    In the context of six-dimensional homogeneous nearly Kähler manifolds, we prove that $\Ss^6$ is the only ambient space admitting constant sectional curvature hypersurfaces.
    In order to do so, we prove first that in $\Ss^3\times\Ss^3$, $\C P^3$ and $F(\C^3)$, any hypersurface with constant sectional curvature is $\eta$-quasi umbilical, where $\eta$ is the dual one-form of the Reeb vector field.
    Then, we use the non-existence of such hypersurfaces in these spaces.
    Additionally, we characterize hypersurfaces of six-dimensional nearly Kähler manifolds which are Sasakian, nearly Sasakian, co-Kähler and nearly cosymplectic.
\end{abstract}

\maketitle
\thispagestyle{empty}

\section{Introduction}
\noindent
Nearly Kähler manifolds were introduced by Tachibana in 1959~\cite{tachibana} and their study was significantly expanded by Gray in subsequent years~\cite{gray3,gray2,gray}.
A nearly Kähler manifold is an almost Hermitian manifold -- a Riemannian manifold $(M,g)$ endowed with a compatible almost complex structure~$J$ -- for which the covariant derivative~$\nabla J$ is skew-symmetric, rather than zero, as is the case for Kähler manifolds.
If, in addition, $\nabla J$ is non-degenerate, $M$ is said to be strict nearly Kähler.

The nearly Kähler condition appears naturally, for example, in the context of the standard round sphere $\Ss^6$ endowed with the almost complex structure induced from the cross product on $\R^7$.
This structure is not Kähler -- $\Ss^6$ can never be Kähler since the second de Rham cohomology is trivial -- but it is nearly Kähler, and moreover, homogeneous.
As such, it is one of the four homogeneous six-dimensional nearly Kähler manifolds ($\Ss^6$, $\Ss^3\times\Ss^3$, $\C P^3$, and $F(\C^3)$), as shown by Butruille~\cite{butruille}.
These manifolds comprise one of the three classes that appear in Nagy's splitting theorem~\cite{nagy1}, which showcases their importance in nearly Kähler geometry.

Nearly Kähler manifolds are abundant in geometry and play an important role in areas seemingly unrelated to nearly Kähler geometry. 
For instance, nearly Kähler structures appear naturally on twistor spaces over quaternionic Kähler manifolds~\cite{ivanov}.
Additionally, Riemannian cones over nearly Kähler manifolds are known to have Riemannian holonomy included in the exceptional Lie group~$G_2$.
In fact, some of the first examples of manifolds with holonomy~$G_2$ were constructed as cones over six-dimensional homogeneous nearly Kähler manifolds~\cite{bryant}.
Moreover, it is noteworthy that six-dimensional strict nearly Kähler manifolds are Einstein.
One fifth of the Einstein constant is known as the type, which we assume to be equal to one in this work.

The structure of these spaces can be better understood by studying their submanifolds, in particular those submanifolds that interact nicely with the structure of the ambient space, e.g.\ the almost complex structure or the isometry group.
When restricting to codimension one in a homogeneous space, a natural class of submanifolds is given by the extrinsically homogeneous hypersurfaces.
These are orbits of a Lie subgroup of the connected component of the identity of the isometry group of the ambient space.
In six-dimensional homogeneous nearly Kähler manifolds, such hypersurfaces were classified~\cite{podesta1}, and are given by the following four families of immersions, parametrized by $t$:
\begin{align}
        \Ss^5\to \Ss^6
        &:p\mapsto \iota_t(p), \label{famS6}\\
        \Ss^2\times\Ss^3\to \Ss^6
        &: (p,q)\mapsto (\cos (t) p,\sin (t) q) \in\R^7,\notag\\
        \Ss^2 \times \Ss^3\to \Ss^3\times\Ss^3
        &: (p,q)\mapsto(\cos (t) \mathbf{1}+\sin (t) p,q), \notag\\
        \Ss^2\times\Ss^3\cong (\Ss^3\times\Ss^3)/\Ss^1\to \C P^3
        &: [p,q]\mapsto [\cos (t) p,\sin (t) q]. \notag
\end{align}
Here, $\iota_t$ for $t \in [0,1)$ denotes a totally umbilical $\Ss^5$ in $\Ss^6$ at height $t$.
The associated Lie subgroups are, respectively, $\SU(3)\subset G_2$, $\SO(4)\subset G_2$, $\SU(2)\times\SU(2) \subset \SU(2)\times\SU(2)\times\SU(2)$, and $\Sp(1)\times\Sp(1) \subset \Sp(2)$. 
In addition, all these immersions are Hopf hypersurfaces~\cite{bolton,hu,michael}.

One of the main results of the current work concerns the immersion~\eqref{famS6}, which is shown to be the only possible way to immerse a space form into a homogeneous six-dimensional nearly Kähler manifold.
\begin{theoremalpha}
    \label{maintheorem}
    Let $M$ be a hypersurface with constant sectional curvature of a six-dimensional homogeneous strict nearly Kähler manifold. 
    Then $M$ is a hypersurface of $\Ss^6$ with sectional curvature greater than or equal to one.
    If the curvature is greater than one, then $M$ is locally a totally umbilical $\Ss^5$.
    If the curvature is equal to one and $M$ is complete, then $M$ is a totally geodesic $\Ss^5$.
\end{theoremalpha}

Almost contact (metric) structures are the odd-dimensional analogue of almost Hermitian structures. 
As such, said structure consists of an odd-dimensional Riemannian manifold $(M,g)$, a unit vector field~$\xi$ (known as the Reeb vector field), the dual one-form $\eta=\xi\contr g$ and a tensor $\varphi$ that satisfies certain compatibility properties (see Equation~\eqref{compatibilitypropalmostcontac}).

A hypersurface of an almost Hermitian manifold $(\tilde{M},g,J)$ with unit normal vector field $N$
carries a natural almost contact structure given by $\xi=-JN$, and $\varphi=J-\eta\otimes N$.
If, in addition, the ambient space is nearly Kähler, 
by plugging in the tensor $\varphi$ in the nearly Kähler condition, we obtain two equations, Equations~\eqref{blairsequations}, that involve $\varphi$ and the second fundamental form.
From these, we deduce that a natural class of hypersurfaces consists of those for which the Reeb vector field $\xi$ is a Killing vector field.
In general, such almost contact manifolds are called almost $K$-contact.
In a nearly Kähler ambient space, almost $K$-contact hypersurfaces are equivalent to hypersurfaces that satisfy
$S\varphi=\varphi S$, where $S$ is the shape operator associated to a unit normal~\cite{blair}.

It has been shown that the homogeneous nearly Kähler $\C P^3$ and $F(\C^3)$ do not admit hypersurfaces that satisfy $S\varphi = \varphi S$~\cite{loubeau}.
In contrast, such hypersurfaces do exist in $\Ss^3 \times \Ss^3$, and have been classified~\cite{hu}.
Finally, in $\Ss^6$, it has been shown that almost $K$-contact hypersurfaces -- hence, satisfying the commutation property -- are the totally umbilical hyperspheres~\cite{kenedy}.
A consequence of these classification results in the six-dimensional homogeneous nearly Kähler spaces, is that any $\eta$-quasi-umbilical hypersurface in such an ambient space must be congruent to the image of a totally umbilical embedding of $\Ss^5$ in $\Ss^6$.

In what follows, we list some types of almost contact manifolds with a Killing Reeb vector field.
\begin{enumerate}[(a)]
    \item 
    A first example is given by Sasakian manifolds: almost contact manifolds $(M,g,\xi,\varphi)$ that satisfy $(\nabla_X\varphi)Y=g(X,Y)\xi-\eta(Y)X$.
    A classical example of a Sasakian manifold is the odd-dimensional sphere $\Ss^{2n+1}(1)$ with the almost contact structure induced from the immersion into $\mathbb{C}^{n+1}$.
    Equivalently, Sasakian manifolds are those manifolds whose Riemannian cone is Kähler, so we can view Sasakian manifolds as odd-dimensional counterparts of Kähler manifolds.

    \item 
    Another possible analogue of Kähler manifolds are the so-called co-Kähler manifolds, those almost contact manifolds $(M,g,\xi,\varphi)$ that satisfy $\nabla\varphi=0$. 
    Originally, they were introduced as cosymplectic manifolds~\cite{bookblair},  but later on the term was re-coined as co-Kähler~\cite{li}.
    Equivalently, a manifold $M$ is co-Kähler if $M\times\R$ is Kähler, hence, a classical example of a co-Kähler manifold is $\R^{2n+1}$.

    \item
    We can define odd-dimensional counterparts of nearly Kähler manifolds, which, in the direction of Sasakian manifolds, are called nearly Sasakian manifolds.
    Such a manifold is an almost contact manifold $(M,g,\xi,\varphi)$  that satisfies 
    $(\nabla_X\varphi)Y+(\nabla_Y\varphi)X=2g(X,Y)\xi-\eta(Y)X-\eta(X)Y$.
    It has been shown that nearly Sasakian hypersurfaces are $\eta$-quasi-umbilical~\cite{blair}.
    It is also noteworthy that the concept of nearly Sasakian manifolds is only of interest in dimensions less than or equal to five, since it has been shown that for other dimensions, such manifolds are automatically Sasakian~\cite{dileo}.

    \item
    In a similar fashion, another analogue of nearly Kähler manifolds is provided by nearly cosymplectic manifolds.
    Such a manifold $(M,g,\xi,\varphi)$ is characterized by $\nabla\varphi$ being skew-symmetric.
\end{enumerate}


On a hypersurface of a given nearly Kähler manifold $\tilde M$ of constant type equal to one, we can define three distinguished almost contact structures: $(\xi=-JN,\phi_i)$ where 
$\phi_1=\varphi$, $\phi_2=\xi \,\lrcorner\,G$, $\phi_3=\phi_2\phi_1=N\,\lrcorner\, G=\nabla\xi-\phi_1S$ and $G=\tilde\nabla J$.
For the totally geodesic $\Ss^5$ in $\Ss^6$, the almost contact structures $\phi_1$ and $\phi_2$ are nearly cosymplectic and $\phi_3$ is Sasakian~\cite{dileo2}.

Recently, a new approach has been proposed for the classification --~formerly presented in~\cite{chinea}~-- of almost contact metric manifolds, in terms of the $(1,1)$-tensor $\tfrac{1}{2}\mathcal{L}_\xi\varphi$~\cite{ilkagiulia}.
In our context, this tensor is often related to the shape operator of the immersion, depending on the class of almost contact manifold.

In the current article, we continue the study of hypersurfaces in the four six-dimensional homogeneous nearly Kähler manifolds. 
In particular, given a hypersurface $M$, any almost contact structure with Reeb vector field $\xi=-JN$ is of the form 
$\phi=a \phi_1+b\phi_2+c\phi_3$ 
with $a^2+b^2+c^2=1$. 
In the following theorem we explore different properties that $(M,g,\xi,\phi)$ may have.
\begin{theoremalpha}
    \label{theoalmostcontact}
    Let $M$ be a hypersurface of a six-dimensional strict nearly Kähler manifold $(\tilde{M},g,J)$ of constant type equal to one,  with unit normal vector field $N$, shape operator $S$ and almost contact structure $(\xi,\phi)$, where $\xi=-JN$ and $\phi = a \phi_1 + b \phi_2 + c \phi_3$ with $a^2 + b^2 + c^2 = 1$.
    Then the following statements hold.
    \begin{enumerate}
        \item $(M,g,\xi,\phi)$ is Sasakian if and only if $S=0$ and $\phi=\phi_3=N\contr G$. \label{theosas}
        \item $(M,g,\xi,\phi)$ is never co-Kähler. \label{theocoK}
        \item If $(M,g,\xi,\phi)$ is nearly Sasakian, then $M$ is $\eta$-quasi-umbilical and belongs to one of the following families: \label{theonS}
        \begin{enumerate}
            \item $S=\id+(\lambda-1)\eta \otimes \xi$, $\lambda\in\R$ and $\phi=\phi_1=J-\eta \otimes N$,  \label{blairscase}
            \item $S=\lambda \id$, $\lambda\in\R$ and $\phi=\phi_3=N\contr G$, \label{casetotumb}
            \item $S=\sec (t)\id$ and $\phi=\cos t \phi_1+\sin t \phi_2$, \label{secantcase}
            \item $S=\frac{1-c}{a}\id+\frac1c \xi(b)\eta\otimes\xi$, where $a,c\ne 0$, $a$ is constant, and  $X(b)=X(c)=0$ for any $X \in \ker\eta$. \label{nSnonconstantcase}
        \end{enumerate}
        \item If  $(M,g,\xi,\phi)$ is nearly cosymplectic, then $M$ is $\eta$-quasi-umbilical and belongs to one of the following families: \label{theonCos}
        \begin{enumerate}
            \item $S=\lambda \id$, $\lambda\in\R$ and $\phi=\phi_2$, \label{caseumbncos}
            \item $S=\lambda \eta\otimes \xi$, $\lambda\in\R$ and $\phi=\phi_1$, \label{casequasncos}
            \item $S=0$ and $\phi=\cos t \phi_1+\sin t \phi_2$, \label{casetotgeoncos}
            \item $S=-\frac{c}{a}\id+\frac1c\xi(b)\eta\otimes\xi$, where $a,c\ne 0$, $a$ is constant, and $X(b)=X(c)=0$ for any $X \in \ker\eta$. \label{casefreencos}
        \end{enumerate}
    \end{enumerate}
\end{theoremalpha}
In a recent unpublished article~\cite{juanma-alberto}, it was shown that the only six-dimensional strict nearly Kähler manifold that admits a totally geodesic hypersurface is the round sphere $\Ss^6$. 
Hence, the only hypersurface admitting a Sasakian structure $(\xi=-JN, \phi)$ is a totally geodesic $\Ss^5\subset\Ss^6$.

\subsection*{Acknowledgments}
The authors would like to thank Dr.\ Kristof Dekimpe for his contributions to this work and Prof.\ Luc Vrancken, Prof.\ Ines Kath and Prof.\ Joeri Van der Veken for their helpful suggestions and guidance.

\section{Preliminaries}

\subsection{Isometrically immersed hypersurfaces}
Let $M$ be an isometrically immersed hypersurface of a Riemannian manifold $(\tilde{M},g)$.
Let $(\nabla,R)$ and $(\tilde{\nabla},\tilde R)$ be their Levi-Civita connection and associated Riemann curvature tensor, respectively, and denote by $N$ a local unit normal vector field.

The connection of the ambient space $\tilde M$ splits in normal and tangent parts with respect to the hypersurface $M$, as the so-called Gauss and Weingarten formulas indicate:
\[
\tilde\nabla_XY=\nabla_XY+h(X,Y)N, \qquad \tilde{\nabla}_XN=-SX,
\]
where $h$ is the second fundamental form and $S$ is the shape operator associated to $N$.
These tensors are related by $h(X,Y)=g\left( SX,Y \right)$ and are both symmetric.
Like the Levi-Civita connection, the curvature $\tilde R$ of the ambient space splits in a normal and tangent part with respect to the hypersurface.
Said tangent part is encoded in the Gauss equation, which involves the curvature $R$ of the hypersurface and the second fundamental form $h$:
\begin{equation}
\label{eq:gauss-equation}
    \begin{aligned}
     g( \tilde{R}(X,Y)Z,W ) &= g( R(X,Y)Z,W ) - h(X,W)h(Y,Z) + h(X,Z)h(Y,W).
    \end{aligned}   
\end{equation}
By taking a cyclic sum of the Gauss equation~\eqref{eq:gauss-equation} over $X$, $Y$ and $W$, we obtain
\begin{equation*}
    \cycl{X,Y,W}\ g(\tilde{R}(X,Y)Z,SW )=\cycl{X,Y,W}\ g( R(X,Y)Z,SW ).
\end{equation*}
If the submanifold $M$ has constant sectional curvature, this equation becomes
\begin{equation}
\label{lucequation}
    \cycl{X,Y,W}\ g( \tilde{R}(X,Y)Z,SW ) = 0.
\end{equation}
It is noteworthy that this formula contains no quadratic terms that involve the second fundamental form, which makes it particularly useful.
A generalization of this formula to higher codimension is sometimes referred to as the \emph{Tsinghua Principle} in the literature.

\begin{definition}
    Let $M$ be an isometrically immersed hypersurface in $(\tilde M,g)$ with second fundamental form $h$.
    We say that $M$ is 
    \begin{enumerate}[(1)]
        \item $u$-quasi-umbilical if $h = \alpha g + \beta u \otimes u$ for some $\alpha, \beta \in C^\infty(M)$ and $u \in \Gamma(T^*M)$,
        \item totally umbilical if $h = \alpha g$ for some $\alpha \in C^\infty(M)$,
        \item totally geodesic if $h = 0$.
    \end{enumerate}
\end{definition}

\subsection{Nearly Kähler manifolds}
A {nearly Kähler} manifold is an almost Hermitian manifold $(\tilde M,g,J)$ for which the covariant derivative of $J$, which we will denote by~$G$, is skew-symmetric.
We say the $\tilde M$ is {strict nearly Kähler} if, in addition,~$G$ is non-degenerate.
The properties of the tensor~$G$ are similar to those of a cross product, as is illustrated by the following proposition.
\begin{proposition}
    \label{prop:G-properties}
    Let $(\tilde M,g,J)$ be a nearly Kähler manifold.
    Then the tensor $G = \tilde\nabla J$ satisfies the following properties:
    \begin{enumerate}
        \item $G(X,JY)+JG(X,Y)=0$,
        \item $g(G(X,Y),Z)+g(G(X,Z),Y)=0$,
        \item $g(G(X,Y),JZ)+g(G(X,Z),JY)=0$.
    \end{enumerate}
\end{proposition}

Six-dimensional nearly Kähler manifolds are Einstein.
That is, the Ricci tensor is a constant multiple of the metric.
If we denote the so-called Einstein constant by $5\alpha$, then the following identities hold in any six-dimensional nearly Kähler manifold:
\begin{align}
        ||G(X,Y)||^2
        &=\alpha \left(g(X,X)g(Y,Y)-g(X,Y)^2-g(JX,Y)^2\right), \label{eq:norm-G-squared}\\
        g(G(X,Y),G(Z,W))
        &=\alpha (g(X,Z)g(Y,W)-g(X,W)g(Y,Z) \label{constanttype}\\
        &\qquad+g(JX,Z)g(Y,JW)-g(JX,W)g(Y,JZ)), \notag\\
        (\tilde\nabla G)(X,Y,Z) 
        &= \alpha\left(g(Y, JZ)X + g(X,Z)JY - g(X, Y )JZ\right),\label{eq:nablaG}\\
        G(G(X, Y ),Z)
        &= \alpha(g(X,Z)Y - g(Y,Z)X - g(JX,Z)JY + g(JY,Z)JX). \label{eq:propg5}
\end{align}
For a nearly Kähler manifold of any dimension satisfying \eqref{constanttype}, the function $\alpha$ is known as the type of $M$, and it is always a positive constant for six-dimensional strict nearly Kähler manifolds~\cite{ilka}.
\begin{theorem}[Butruille, \cite{butruille}]
\label{thm:butruille}
    A simply connected, complete, homogeneous six-dimensional strict nearly Kähler manifold is, up to homotheties, isometric to $\Ss^6$, $\Ss^3\times\Ss^3$, $F(\C^3)$ or $\C P^3$.
\end{theorem}

\subsection{Almost contact structures}
An almost contact (metric) structure consists of an odd-dimensional Riemannian manifold $(M,g)$, a unit vector field~$\xi$ (known as the Reeb vector field), the dual one-form $\eta=\xi\contr g$ and a tensor $\varphi$ that satisfies
\begin{equation}
g(\varphi X,\varphi Y)=g(X,Y)-\eta(X)\eta(Y), \qquad \varphi^2=-\id+\eta\otimes\xi. \label{compatibilitypropalmostcontac}    
\end{equation}
As a consequence, $\varphi$ and $\eta$ satisfy $\varphi(\xi)=0$ and $\eta\circ\varphi=0$.

Now assume that $M$ is a hypersurface of an almost Hermitian manifold $(\tilde{M},g,J)$ with unit normal vector field $N$, second fundamental form $h$ and shape operator $S$. 
Then $M$ carries a natural almost contact structure given by
\begin{equation}
\xi=-JN,  \qquad \varphi=J-\eta\otimes N. \label{phi1}  
\end{equation}

If, in addition, the ambient space is nearly Kähler, then, by plugging the expressions~\eqref{phi1} into the condition that $\tilde\nabla J$ is skew-symmetric, we obtain the following conditions on $\varphi$, $h$ and $S$:
\begin{equation}
    \begin{split}
        (\nabla_X\varphi)Y+(\nabla_Y\varphi)X+2h(X,Y)\xi-\eta(X)SY-\eta(Y)SX&=0, \\
        h(\varphi X,Y)+h(X,\varphi Y)+g(\nabla_X\xi,Y)+g(\nabla_Y\xi,X)&=0.
    \end{split}\label{blairsequations}
\end{equation}
From these equations, we note that a natural class of almost contact hypersurfaces to study are the ones whose Reeb vector field $\xi$ is Killing, known as almost $K$-contact manifolds.
Such hypersurfaces are equivalent with hypersurfaces that satisfy
$S\varphi=\varphi S$, where $S$ is the shape operator associated to $N$.

\begin{definition}
    Let $(M,g,\xi,\varphi)$ be an almost contact manifold.
    We say that $M$ is
    \begin{enumerate}
        \item Sasakian if $(\nabla_X\varphi )Y=\eta(Y)X-g(X,Y)\xi$ for all $X,Y\in\mathfrak{X}(M)$,
        \item co-Kähler if $\nabla\varphi\equiv0$,
        \item nearly Sasakian if $(\nabla_X\varphi)Y+(\nabla_Y\varphi)X=\eta(Y)X+\eta(X)Y-2g(X,Y)\xi$  for all $X,Y\in\mathfrak{X}(M)$. 
        \item nearly cosymplectic if $(\nabla_X\varphi)Y+(\nabla_Y\varphi)X=0$ for all $X,Y\in\mathfrak{X}(M)$,
        
    \end{enumerate}\label{definitionsalmostcontact}
\end{definition}
Note that by taking $-\xi$ as the Reeb vector field instead of $\xi$, we obtain an equivalent definition of (nearly) Sasakian, often found in the literature.

\section{Almost contact structures on hypersurfaces}
\noindent
We prove Theorem~\ref{theoalmostcontact} in this section.
In order to do so, we need the covariant derivatives of the almost contact structures $\{(\xi,\phi_i)\}_{i=1,2,3}$ on a hypersurface of a nearly Kähler manifold of constant type $\alpha=1$.
Recall the definitions
\[
    \phi_1 = J-\eta\otimes N, \qquad 
    \phi_2 = \xi\contr G, \qquad
    \phi_3 = \phi_2\phi_1 = N\contr G = \nabla\xi-\phi_1S.
\]
We denote by $\omega_i$ the respective fundamental forms of the tensors $\phi_i$, that is, $\omega_i(X,Y)=g(\phi_i X,Y)$.
A direct computation using the Gauss and Weingarten formulas, Proposition~\ref{prop:G-properties}, and Equations~\eqref{eq:nablaG} and \eqref{eq:propg5}, yields
\begin{equation}
    \begin{split}
            (\nabla_X \phi_1) Y&=G(X,Y)-\omega_3(X,Y)N+\eta(Y)SX-h(X,Y)\xi, \\
            (\nabla_X \phi_2) Y&= \eta(Y)\phi_1 X-\eta(X)\phi_1 Y-\omega_1(X,Y)\xi \\
            &\quad+\omega_2(SX,Y)N+G(\phi_1SX,Y)+g(X,S\xi)\phi_3Y, \\
            (\nabla_X \phi_3) Y&=\eta(Y)X-g(X,Y)\xi-G(SX,Y)+\omega_3(SX,Y)N.
    \end{split}\label{covariant derivatives}
\end{equation}

\begin{proof}[Proof of Theorem~\ref{theoalmostcontact}]
    Recall that (nearly) Sasakian, co-Kähler and nearly cosymplectic manifolds are almost $K$-contact.
    Moreover, recall that any almost contact structure on $M$ with Reeb vector field $\xi=-JN$ is of the form $\phi=a \phi_1+b \phi_2+c\phi_3$ with $a^2+b^2+c^2=1$.
    
    From Equation~\eqref{blairsequations} it follows that for an almost $K$-contact hypersurface of a nearly Kähler manifold, $\xi$ is an eigenvector of the shape operator $S$.
    Now let $\zeta$ be a unit eigenvector of $S$ contained in $\ker\eta$.
    This way, $\mathcal{F}=\{\zeta,\phi_1\zeta,\phi_2\zeta,\phi_3\zeta,\xi\}$ is an orthonormal frame on~$M$. 
    If we denote by $h_{ij}$ the components of $S$ with respect to $\mathcal{F}$, then
    Equation~\eqref{blairsequations} yields $h_{ij}=0$ for $i\neq j$ and $h_{11}=h_{22}$, $h_{33}=h_{44}$.

    In what follows, we use the expressions for the covariant derivatives in~\eqref{covariant derivatives} and we calculate the relevant defining equation of the hypersurface (see Definition~\ref{definitionsalmostcontact}) with $X,Y\in\mathcal F$, taking the inner product with $Z \in \mathcal F$.

    \textit{Case~(\ref{theosas}).}
    If we choose $X=\zeta$, $Y=\phi_1\zeta$ and $Z=\xi$, we obtain $b=0$.
    In addition, taking $X,Y,Z \in \{\xi,\zeta,\phi_2\zeta\}$ all different, we obtain $a=ch_{11}=ch_{33}=ch_{55}$, from which it follows that $c$ cannot vanish.
    Hence, $h_{11}=h_{33}=h_{55}$.
    Moreover, choosing $X \in \mathcal F$, $Y=\zeta$ and $Z=\phi_1\zeta$, we see that $a$ and $c$ are constant.
    Combining this with $a^2+c^2=1$ and with $ah_{11}=1-c$, which we obtain from $X,Y=\zeta$ and $Z=\xi$, we conclude that $a=0$, $c=1$ and $h_{11}=h_{33}=h_{55}=0$.

    \textit{Case~(\ref{theocoK}).}
    In a similar way as for Case~(\ref{theosas}), we obtain that $b=0$, $a$ and $c\neq 0$ are constants, $a=ch_{11}$ and $ah_{11}=-c$, which contradict the fact that $a^2+b^2+c^2=1$.

    \textit{Case~(\ref{theonS}).}
    Taking $X=Y\in \mathcal{F}\smallsetminus\{\xi\}$ and $Z\in\{\phi_1 X,\phi_2X,\phi_3X\}$ on the one hand, and $X=\zeta$, $Y=\xi$ and $Z\in\{\phi_1\zeta,\phi_2\zeta,\phi_3\zeta\}$ on the other hand, we see that $a$ is constant and $b$, $c$ are constant in all directions in $\ker(\eta)$, and 
    \begin{equation}
        \xi(b)=c(h_{55}-h_{11}), \qquad \xi(c)=b (h_{11}-h_{55}). \label{derivativesabc}
    \end{equation}
    To continue, we consider two cases: $c=0$ and $c\neq0$.

    Suppose that $c=0$.
    Then $a=\cos t$, $b=\sin t$ where $t$ is a constant and $\sin t(h_{11}-h_{55})=0$ by~\eqref{derivativesabc}. 
    If we plug in $X,Y=\zeta$ and $Z=\xi$ on the one hand, and $X=\xi$ and $Y,Z=\phi_2\zeta$ on the other hand, then we obtain respectively $h_{11}\cos t=1$ and $h_{33}\cos t=1$.
    If $t=0$, then we obtain~(\ref{blairscase}).
    If $t\neq0$, then $h_{11}=h_{33}=h_{55}=\sec t$ and $\phi=\cos t\phi_1+\sin t \phi_2$, which is~(\ref{secantcase}).
    
    Suppose that $c\neq0$.
    Choosing $X=\zeta$, $Y=\phi_2\zeta$ and $Z=\xi$, we obtain $h_{33}=h_{11}$.
    We again consider two cases: $a=0$ and $a\neq0$.
    If $a=0$, then taking $X,Z=\zeta$ and $Y=\xi$, we get $c=1$, $b=0$, and from~\eqref{derivativesabc} we conclude that $h_{55}=h_{11}$.
    This is~(\ref{casetotumb}).
    If $a\neq0$, then $X,Y=\zeta$ and $Z=\xi$ yields $h_{33}=h_{11}=\tfrac{1-c}{a}$, and from~\eqref{derivativesabc} we obtain $h_{55}=\frac{a \xi(b)-c^2+c}{a c}$. 
    This is~(\ref{nSnonconstantcase}).

    \textit{Case~(\ref{theonCos}).}
    Taking $X=Y\in \mathcal{F}\smallsetminus\{\xi\}$ and $Z\in\{\phi_1 X,\phi_2X,\phi_3X\}$ on the one hand, and $X=\zeta$, $Y=\xi$ and $Z\in\{\phi_1\zeta,\phi_2\zeta,\phi_3\zeta\}$ on the other hand, we see that $a$ is constant and $b$, $c$ are constant in all directions in $\ker(\eta)$, and 
    \begin{equation}
        \xi(b)=c(h_{55}-h_{11}),  \qquad \xi(c)=b (h_{11}-h_{55}). \label{derivativesncosabc}
    \end{equation}
    To continue, we consider two cases: $c=0$ and $c\neq0$.

    Suppose that $c=0$.
    Then $a=\cos t$, $b=\sin t$ where $t$ is a constant and $\sin t(h_{11}-h_{55})=0$, by~\eqref{derivativesncosabc}. 
    If we plug in $X,Y=\zeta$ and $Z=\xi$ on the one hand, and $X=\xi$ and $Y,Z=\phi_2\zeta$ on the other hand, then we obtain respectively $h_{11}\cos t=0$ and $h_{33}\cos t=0$.
    If $t=\tfrac{\pi}{2}$, then we obtain~(\ref{caseumbncos}).
    If $t=0$, then $h_{11}=h_{33}=0$ and $\phi=\phi_1$, which yields~(\ref{casequasncos}). 
    If $t\neq0$ or $\tfrac{\pi}{2}$, then $h_{11}=h_{33}=h_{55}=0$ and $\phi=\cos t\phi_1+\sin t \phi_2$, which is~(\ref{casetotgeoncos}).
    
    Suppose that $c\neq0$.
    Plugging in $X,Y=\zeta$ and $Z=\xi$, we obtain $ah_{11}=-c$, therefore $a$ cannot be zero.
    Moreover, the same equation with $X=\zeta$, $Y=\phi_2\zeta$ and $Z=\xi$ yields $h_{11}=h_{33}$.
    From~\eqref{derivativesncosabc} we then obtain $h_{55}=\frac{a \xi(b)-c^2}{a c}$. 
    This is~(\ref{casefreencos}).
\end{proof}

\section{The four six-dimensional homogeneous nearly Kähler spaces}
\noindent
In what follows we describe the nearly Kähler structures and other related tensors on the manifolds that appear in Butruille's classification, Theorem~\ref{thm:butruille}.
As this result is up to homotheties, we choose the metric on each of these spaces in such a way that the type $\alpha$ is equal to one.

\subsection{The sphere \texorpdfstring{$\boldsymbol{\Ss^6}$}{S6}}
Let $\mathbb{O}$ be the standard real octonion algebra. 
We may identify the imaginary octonions $\mathrm{Im}(\mathbb{O})$ with $\R^7$, which admits an inner product $\li x,y\ri=-\tfrac{1}{2}(xy+yx)$ and a cross product $x\times y=\tfrac{1}{2}(xy-yx)$.
We define the six-sphere $\Ss^6$ as the set of unit length imaginary octonions, equipped with the induced metric from $\R^7$, and an almost complex structure~$J$ defined by $J_pX=p\times X$.
With these structures, $(\Ss^6,\li\cdot,\cdot\ri,J)$ is a strict nearly Kähler manifold.
As a homogeneous space, $\Ss^6$ is the quotient $G_2/\SU(3)$.

Since the metric $\li\cdot,\cdot\ri$ is the round metric on $\Ss^6$, the curvature of the associated Levi-Civita connection is given by
$\tilde R(X,Y)Z=\li Y,Z\ri X-\li X,Z\ri Y.$

\subsection{The nearly Kähler \texorpdfstring{$\boldsymbol{\Ss^3\times\Ss^3}$}{S3xS3}}
Let $\mathbb{H}$ be the standard real quaternion algebra. 
We identify the three-sphere $\Ss^3$ with the unit quaternions. 
That is, $\Ss^3=\{p\in\mathbb{H}:\langle p,p\rangle=1\}$.
The tangent space at a point $p\in \Ss^3$ is then given by $T_p\Ss^3=p^\perp=\{p\alpha:\alpha\in\mathrm{Im}(\mathbb{H})\}$.

We define a metric $g$ on $\Ss^3 \times \Ss^3$ by $g((p\alpha,q\beta),(p\gamma,q\delta))=\frac{4}{9}\li (p\alpha,q\beta),(p\gamma,q\delta)\ri-\frac{2}{9}\li(p\beta,q\gamma),(p\gamma,q\delta)\ri$
and an almost complex structure $J$ by
$J(p\alpha,q\beta)=\frac{1}{\sqrt{3}}(p(\alpha-2\beta),q(2\alpha-\beta))$.
One can show that $(\Ss^3\times\Ss^3,g,J)$ is a strict nearly Kähler manifold.
As a homogeneous space, $\Ss^3\times\Ss^3$ is the quotient $\frac{\SU(2)\times\SU(2)\times\SU(2)}{\Delta\SU(2)}$.

The curvature of the Levi-Civita connection associated to~$g$ is, for instance, given in~\cite{lagrS3S3} and is presented as follows:
\begin{equation}
\label{eq:curvatures3s3}
    \begin{aligned}
    \tilde R(X,Y)Z
    &=\frac{5}{4} \big(g(Y,Z)X - g(X,Z)Y\big)\\
    &\quad +\frac{1}{4}
    \big(g(JY,Z)JX - g(JX,Z)JY - 2g(JX, Y )JZ\big)\\
    &\quad + 
    g(PY,Z)PX - g(PX,Z)PY
    + g(JPY,Z)JPX - g(JPX,Z)JPY,
    \end{aligned}
\end{equation}
where $P$ is an almost product structure on $\Ss^3\times\Ss^3$ given by 
\begin{equation}
\label{eq:product-structure-s3xs3}
    P(p\alpha,q\beta)=(p\beta,q\alpha).
\end{equation}
This tensor satisfies the following properties, where we use the notation $G=\tilde \nabla J$:
\begin{equation}
\label{pprop}
    \begin{alignedat}{2}
        &P^2=\id, 
        &&g(PX,PY)=g(X,Y),\\
        &PJ=-JP, \hspace{1.5 cm}
        &&PG(X,Y)=-G(PX,PY).
    \end{alignedat}
\end{equation}

\subsection{The nearly Kähler \texorpdfstring{$\boldsymbol{\C P^3}$}{CP3}}
\label{sectioncp3structure}
It is well known that there exists a Riemannian submersion $\pi:\C P^3\to \Ss^4$, the so-called twistor fibration.
The associated vertical and horizontal distributions are denoted by $\mathcal D_1$ and $\mathcal D_2$, respectively.
Now define the metric $g$ and almost complex structure $J$ as follows: 
\[
    g=\begin{cases}
     g_o & \text{ on }\mathcal{D}_1,\\
     \frac{1}{2} g_o & \text{ on }\mathcal{D}_2,
    \end{cases} \hspace*{1.5cm}
    J=\begin{cases}
        J_o & \text{ on }\mathcal{D}_1, \\
        -J_o & \text{ on }\mathcal{D}_2, \\
    \end{cases}
\]
where $(g_o,J_o)$ is the standard Kähler structure on $\C P^3$.
One can show that $(\C P^3,g,J)$ is a strict nearly Kähler manifold, with homogeneous representation $\frac{\Sp(2)}{\Sp(1)\U(1)}$.

The curvature of the Levi-Civita connection associated to~$g$ is given in \cite{michaelthesis}, and is presented as follows:
\begin{equation}
\label{eq:curvature-cp3}
    \begin{aligned}
        \tilde R(X,Y)Z 
        =& \frac{5}{4} \big(g(Y,Z)X-g(X,Z)Y\big) \\
        &-\frac{1}{4}\big( g(JY,Z)JX - g(JX,Z)JY + 2g(X,JY)JZ \big) \\
        & +\frac{1}{2} \big( g(J_oY,Z)J_oX - g(J_oX,Z)J_oY + 2g(X,J_oY)J_oZ \big) \\
        & +\frac{1}{2} \big( g(JJ_oY,Z)X - g(JJ_oX,Z)Y + g(Y,Z)JJ_oX - g(X,Z)JJ_oY \big) \\
        & + g(JJ_oY,Z)JJ_oX - g(JJ_oX,Z)JJ_oY. \\
    \end{aligned}
\end{equation}

\subsection{The flag manifold \texorpdfstring{$\boldsymbol{F(\C^3)}$}{F(C3)}}
\label{sec:flag-introduction}
The manifold of full flags in $\C^3$ is defined as the reductive homogeneous space $\frac{\SU(3)}{\U(1)\times\U(1)}$.
Indeed, one can show that the Lie algebra $\mathfrak{su}(3) = \mathfrak h \oplus \mathfrak m$ where $\mathfrak h$ is the Lie algebra of $\U(1)\times\U(1):= H$ and $\mathfrak m$ is $\mathrm{Ad}_H$-invariant.
In fact, $\mathfrak m$ splits into three two-dimensional subspaces $\mathfrak m_i\ (i=1,2,3)$ which are each invariant under the adjoint action of $H$.

We may define an almost complex structure $J$ on $\mathfrak m$ that preserves each factor and as metric $g$ on $\mathfrak m$ we may take the restriction of the Killing form on $\mathfrak{su}(3)$. 
Identifying $\mathfrak m$ with $T_oF(\C^3)$, we left extend this almost complex structure and metric to the whole space, thus obtaining a strict nearly Kähler manifold.
Note that the left extension of the subspaces $\mathfrak m_i$ gives rise to three left-invariant distributions $\mathcal D_i$ on $TF(\C^3)$.
For more details, we refer to~\cite{kamilluc}.

Now define, for $i=1,2,3$, the almost complex structures
\begin{equation*}
J_i=\begin{cases}
    J & \text{on $\mathcal{D}_i$},\\
    -J & \text{otherwise}.\\
\end{cases}
\end{equation*}
The almost complex structures $J$ and $J_1,J_2,J_3$ appear in a natural way in the curvature tensor of the Levi-Civita connection associated to the nearly Kähler metric~$g$, namely:
\begin{equation}
\label{eq:flag-curvaturetensor}
    \begin{split}
        \tilde R(X,Y)Z&=\frac{1}{4}(g(Y,Z)X - g(X,Z)Y )\\
        &\quad- \frac{1}{4} (g(J  Y,Z)J  X - g(J  X,Z)J  Y + 2g(X, J  Y )J  Z)\\
        &\quad+ \frac{1}{2} (g(J_1Y,Z)J_1X - g(J_1X,Z)J_1Y + 2g(X, J_1Y )J_1Z)\\
        &\quad+ \frac{1}{2} (g(J_2Y,Z)J_2X - g(J_2X,Z)J_2Y + 2g(X, J_2Y )J_2Z)\\
        &\quad+ \frac{1}{2} (g(J_3Y,Z)J_3X - g(J_3X,Z)J_3Y + 2g(X, J_3Y )J_3Z).
    \end{split}
\end{equation}

\section{Hypersurfaces with constant sectional curvature}
\noindent
In this section we prove Theorem~\ref{maintheorem}.
To do so, we first show that a hypersurface with constant sectional curvature in $\Ss^3\times\Ss^3$, $\C P^3$ or $F(\C^3)$ is $\eta$-quasi-umbilical, where $\eta = JN \contr g$ and $N$ is a unit normal to $M$.
Secondly, we use the non-existence of $\eta$-quasi-umbilical hypersurfaces in these spaces~\cite{loubeau,hu}.
The case of $\Ss^6$ is treated separately.

The essence of the proof of Theorem~\ref{maintheorem} is Equation~\eqref{lucequation}.
In order to calculate $g( \tilde R(X,Y)Z, SW )$, we address each ambient space separately and determine an orthogonal frame $\{E_1,\ldots,E_5\}$ on the hypersurface with respect to which we can easily determine the tensors present in the expression of $\tilde R$.
Since the shape operator $S$ also appears, we denote by $h_{ij} = g(S E_i, E_j)$ its components with respect to said frame.
The aim of calculating Equation~\eqref{lucequation} is to determine these components.

\begin{remark}
\label{remark-hor-lines}
    The result of Equation~\eqref{lucequation} is listed in Table~\ref{tableS3S3independent} up to Table~\ref{table1FC3D1D2D3}.
    A horizontal dotted line indicates that, in order to obtain the equations below the line, the information provided by the equations above has to be used. 
\end{remark}

\textit{Proof of Theorem~\ref{maintheorem}.}
Let $M$ be a hypersurface with constant sectional curvature in a six-dimensional homogeneous nearly Kähler manifold.
We divide our argument into four different cases, depending on the ambient space in which $M$ is immersed: $\Ss^6$, $\Ss^3\times\Ss^3$, $\C P^3$ or $F(\C^3)$.

\subsection*{Hypersurfaces of \texorpdfstring{$\boldsymbol{\Ss^6}$}{S6}}
The shape operator $S$ associated to $N$ diagonalizes with respect to a \emph{smooth} orthogonal frame, see~\cite[Lemma\ 1.2]{szabo}.
By plugging this frame into the Gauss equation~\eqref{eq:gauss-equation}, we conclude that either $M$ has sectional curvature $c$ greater than one and is an open subset of a totally umbilical~$\Ss^5$,
or that $c=1$.
In the latter case, if, in addition, $M$ is complete, it follows from Corollary 7.13 in~\cite{tojeirobook} that $M$ is a totally geodesic~$\Ss^5$.

\subsection*{Hypersurfaces of \texorpdfstring{$\boldsymbol{\Ss^3\times\Ss^3}$}{S3xS3}}
Recall that $\Ss^3 \times \Ss^3$ carries an almost product structure $P$, given by Equation~\eqref{eq:product-structure-s3xs3}.
We divide our argument into two cases: we consider the case where $PN$ is linearly independent from~$N$ and~$JN$, and the case where $PN$ is a linear combination of~$N$ and~$JN$.

\begin{lemma}
    Let $M$ be a hypersurface of $\Ss^3\times\Ss^3$ with constant sectional curvature.
    Let $N$ be its unit normal and suppose that $PN$ is linearly independent from $N$ and $JN$.
    Then $M$ is totally umbilical.
\end{lemma}
\begin{proof}
We define the vector field
$V=PN-\theta_1N-\theta_2JN$,
where $\theta_1=g(PN,N)$ and $\theta_2=g(PN,JN)$.
From a direct computation 
we see that $\mathcal F = \{JN,V,JV,G(V,N),JG(V,N)\}$ is an orthogonal frame on $M$.
Moreover, $JN$ is a unit vector field and
$g(V,V)=g( G(V,N),G(V,N))=1-\theta_1^2-\theta_2^2$.
Note that $PN$ being linearly independent from $N$ and $JN$ implies $\theta_1^2+\theta_2^2\neq1$.
Then, it follows from~\eqref{pprop} that the almost product structure $P$ is given by the following matrix with respect to $\mathcal F\cup \{N\}$:
\begin{equation}
\label{eq:s3-times-s3-P-components}
    \begin{pmatrix}
         -\theta _1 & 0 & \theta _1^2+\theta _2^2-1 & 0 & 0 & \theta_2 \\
         0 & -\theta _1 & \theta_2 & 0 & 0 & 1 \\
        -1 & \theta_2 & \theta _1 & 0 & 0 & 0 \\
         0 & 0 & 0 & 1 & 0 & 0 \\
         0 & 0 & 0 & 0 & -1 & 0 \\
         \theta_2 & 1-\theta_1^2-\theta_2^2& 0 & 0 & 0 & \theta_1 \\
    \end{pmatrix}.
\end{equation}

We proceed by computing Equation \eqref{lucequation}, using~\eqref{eq:s3-times-s3-P-components} in the curvature~\eqref{eq:curvatures3s3}, plugging in the vector fields $X$, $Y$, $Z$ and $W$ as indicated in Table~\ref{tableS3S3independent} (see Remark~\ref{remark-hor-lines}).
The result of the computation is listed in the right column.
From these equations it follows that $h_{11}(1-\theta_1^2-\theta_2^2) = h_{22} = h_{33} = h_{44} = h_{55}$.
Moreover, combining the last two equations, we find that $h_{34}=0$.
Hence, when taking into account the norms of the elements of $\mathcal F$, we see that $M$ is totally umbilical with shape operator $S=h_{11}\id$.
\end{proof}

\begin{lemma}
    Let $M$ be a hypersurface of $\Ss^3\times\Ss^3$ with constant sectional curvature.
    Let $N$ be its unit normal and suppose that $PN$ is linearly dependent on $N$ and $JN$.
    Then $M$ is $\eta$-quasi-umbilical, where $\eta$ is the dual 1-form of $JN$.
\end{lemma}
\begin{proof}
    Since $P$ is compatible with the metric, $PN$ has unit length and therefore $PN=\cos \theta N+\sin \theta JN$ for some $\theta \in C^\infty(M)$.
    Now the tangent bundle of $\Ss^3\times\Ss^3$ splits in two $P$-invariant distributions, that is, $T (\Ss^3\times\Ss^3)=\mathcal{D}_1\oplus\mathcal{D}_2$ where $\mathcal{D}_1=\mathrm{span}\{N,JN\}$ and $\mathcal{D}_2=\mathcal{D}_1^\perp$. 
    Since $\mathcal{D}_2$ is $P$-invariant, it is generated by the orthonormal set of smooth vector fields $\{V_1,V_2,JV_1,JV_2\}$, where $V_1$ and $V_2$ 
    at each point are eigenvectors of $P|_{TM}$ associated to $1$ and $-1$, respectively.
    As $P$ anti-commutes with $J$, the vector fields $JV_1$ and $JV_2$ are eigenvectors of $P|_{TM}$ associated to $-1$ and $1$, respectively.
    Hence, $\mathcal{F}=\{JN,V_1,V_2,JV_1,JV_2\}$ is an orthonormal frame on $M$.

    We proceed by computing Equation~\eqref{lucequation}, using the curvature~\eqref{eq:curvatures3s3}, and taking the vector fields $X$, $Y$, $Z$ and $W$ as indicated in Table~\ref{tableS3S3casedependent} (see Remark~\ref{remark-hor-lines}).
    The result of the computation is listed in the right column.
    From these calculations it follows that the shape operator with respect to $\mathcal F$ is diagonal.
    Moreover, from the first three equations we deduce that $h_{22}=h_{33}=h_{44}=h_{55}$.
\end{proof}

\subsection*{Hypersurfaces of \texorpdfstring{$\boldsymbol{\C P^3}$}{CP3}}
Recall that the tangent bundle of $\C P^3$ splits as $\mathcal D_1 \oplus \mathcal D_2$ (see Section~\ref{sectioncp3structure}), where $\mathrm{rank}\ \mathcal D_1 = 2$ and $\mathrm{rank}\ \mathcal D_2 = 4$.
We divide our argument into three cases: $N\in\mathcal{D}_1$, $N\in \mathcal{D}_2$ and $N=V_1+V_2$ where $V_i\in \mathcal{D}_i\smallsetminus\{0\}$, $i=1,2$. 

\begin{lemma}
    Let $M$ be a hypersurface of $\C P^3 $ with constant sectional curvature. 
    Let $N$ be its unit normal and suppose that $N\in\mathcal{D}_1$. 
    Then $M$ is $\eta$-quasi-umbilical, where $\eta$ is the dual 1-form of $JN$.
\end{lemma}
\begin{proof}

Let $V$ be a unit vector field in $\mathcal D_2$.
Then a direct computation using~\eqref{eq:norm-G-squared} with $\alpha = 1$, and the fact that $J$ is compatible with $g$, shows that 
$\mathcal F = \{JN, V,JV, G(V,N), JG(V,N)\}$
is an orthonormal frame on $M$.
Note that $\mathcal D_1 = \mathrm{span}\{N,JN\}$ and $\mathcal D_2 = \mathrm{span}\{V,JV,G(V,N),JG(V,N)\}$.

To compute Equation~\eqref{lucequation}, we need to express the curvature $\tilde R$ as well as the shape operator with respect to $\mathcal F$.
For this purpose, we determine how $J$ and $J_o$ act on the elements of $\mathcal F$, allowing us to calculate $\tilde R$ given by~\eqref{eq:curvature-cp3}.

We proceed by computing Equation~\eqref{lucequation}, using the curvature~\eqref{eq:curvature-cp3} and plugging in the vector fields $X$, $Y$, $Z$ and $W$ as indicated in Table~\ref{tableCP3D1} (see Remark~\ref{remark-hor-lines}).
The result of the computation is listed in the right column.
From these calculations it follows that the shape operator with respect to $\mathcal F$ is diagonal.
Moreover, from the first three equations we deduce that $h_{22}=h_{33}=h_{44}=h_{55}$.
\end{proof}

\begin{lemma}
    Let $M$ be a hypersurface of $\C P^3 $ with constant sectional curvature. 
    Let $N$ be its unit normal and suppose that $N\in\mathcal{D}_2$. 
    Then $M$ is $\eta$-quasi-umbilical where $\eta$ is the dual 1-form of $JN$.
\end{lemma}
\begin{proof}
Let $X$ be a unit vector field in $\mathcal D_1$.
Then a direct computation using~\eqref{eq:norm-G-squared} with $\alpha = 1$, and the fact that $J$ is compatible with $g$ shows that 
$\mathcal F = \{JN, V, JV, G(V,N), JG(V,N)\}$
is an orthonormal frame on $M$.
Note that $\mathcal D_1 = \mathrm{span}\{V,JV\}$ and $\mathcal D_2 = \mathrm{span}\{N,JN,G(V,N),JG(V,N)\}$.

To compute Equation~\eqref{lucequation}, we need to express the curvature $\tilde R$ as well as the shape operator with respect to $\mathcal F$.
To this end, we determine how $J$ and $J_o$ act on the elements of $\mathcal F$, allowing us to calculate $\tilde R$ given by~\eqref{eq:curvature-cp3}.
Then we take $X$, $Y$, $Z$ and $W$ as indicated in Table~\ref{tableCP3D2} (see Remark~\ref{remark-hor-lines}) and the result of the computation is listed in the right column.
From these calculations it follows that the shape operator with respect to $\mathcal F$ is diagonal.
Moreover, from the first three equations we deduce that $h_{22}=h_{33}=h_{44}=h_{55}$.
\end{proof}

\begin{lemma}
    Let $M$ be a hypersurface of $\C P^3 $ with constant sectional curvature. 
    Let $N$ be its unit normal and suppose that $N=V_1+V_2$, where $V_i\in\mathcal D_i\smallsetminus\{0\}$.
    Then $M$ is totally umbilical.
\end{lemma}

\begin{proof}
    Note that from the properties of nearly Kähler manifolds it follows that $\mathcal{D}_2=\mathfrak{D}_1\oplus\mathfrak{D}_2$, where $\mathfrak{D}_1=\mathrm{span}\{V_2,JV_2\}$ and $\mathfrak{D}_2=\mathrm{span}\{G(V_1,V_2),JG(V_1,V_2)\}$.
    To form an orthogonal frame on $M$, we define the vector field $V_3=G(G(V_1,V_2),N)$. 
    By~\eqref{eq:propg5} we get $V_3=\cos^2\theta N-V_1$, where $\cos^2\theta=g(V_1,N)$, with $\theta\in(0,\tfrac{\pi}{2})$.
    Note that the vector field $V_3$ is orthogonal to $N$, hence tangent to $M$.
    In addition, we see that $V_3$ is orthogonal to $JV_1$, $JV_2$, $G(V_1,V_2)$ and $JG(V_1,V_2)$, which are all vector fields tangent to $M$.
    Thus we obtain an orthogonal frame $\mathcal F = \{JV_1,JV_2,V_3,G(V_1,V_2),JG(V_1,V_2)\}$ on $M$ where $g(JV_1,JV_1)=\cos^2\theta$, $g(JV_2,JV_2)=\sin^2\theta$, and $g(V_3,V_3)=g(G(V_1,V_2),G(V_1,V_2))=g(JG(V_1,V_2),JG(V_1,V_2))=\cos^2\theta\sin^2\theta$.
    It is a straightforward computation to determine how $J$ and $J_o$ behave with respect to~$\mathcal F$. 
 
    We proceed by computing Equation~\eqref{lucequation}, using the curvature~\eqref{eq:curvature-cp3} and plugging in the vector fields $X$, $Y$, $Z$ and $W$ as indicated in Table~\ref{tableCP3D1D2} (see Remark~\ref{remark-hor-lines}).
    The result of the computation is listed in the right column.
    From these equations we deduce that the shape operator is diagonal with respect to $\mathcal F$.
    Moreover, the last three equations imply that $h_{22}\cos^2\theta=h_{33}=h_{44}=h_{55}$ and hence $M$ is totally umbilical with $S=h_{11}\sec^2\theta \id$.
\end{proof}

\subsection*{Hypersurfaces of \texorpdfstring{$\boldsymbol{F(\C^3)}$}{F(C3)}}
Recall that the tangent bundle of $F(\C^3)$ splits as $TF(\C^3)=\mathcal{D}_1\oplus \mathcal{D}_2 \oplus \mathcal{D}_3$ (see Section~\ref{sec:flag-introduction}).
We divide in three cases: $N$ lies in one distribution (which we may assume to be $\mathcal{D}_1$), $N$ has non-zero components in two of the distributions (which we may assume to be $\mathcal{D}_1$ and $\mathcal{D}_2$), and $N$ has non-zero components in all three distributions.

\begin{lemma}
    Let $M$ be a hypersurface of $F(\C^3)$ with constant sectional curvature. 
    Let $N$ be its unit normal and suppose that $N\in\mathcal{D}_1$. 
    Then $M$ is $\eta$-quasi-umbilical where $\eta$ is the dual 1-form of~$JN$.
\end{lemma}
\begin{proof}
    Let $U$ and $V$ be unit vector fields in $\mathcal{D}_2$ and $\mathcal{D}_3$, respectively.
    Then $\mathcal F = \{JN,U,JU,V,JV\}$ is a orthonormal frame on $M$. 
    Moreover, it is easy to see how $J$, $J_1$, $J_2$ and $J_3$ behave with respect to~$\mathcal F$.

    We proceed by computing Equation~\eqref{lucequation}, using the curvature~\eqref{eq:flag-curvaturetensor} and plugging in the vector fields $X$, $Y$, $Z$ and $W$ as indicated in Table~\ref{tableFC3D1} (see Remark~\ref{remark-hor-lines}).
    The result of the computation is listed in the right column.
    From these computations it follows that the shape operator is diagonal with respect to $\mathcal F$.
    Moreover, from the first three equations we deduce that $h_{22}=h_{33}=h_{44}=h_{55}$.
\end{proof}

\begin{lemma}
    Let $M$ be a hypersurface of $F(\C^3)$ with constant sectional curvature. 
    Let $N$ be its unit normal and suppose that $N$ has non-zero components in $\mathcal D_1$ and $\mathcal D_2$, but no components in $\mathcal D_3$.
    Then $M$ is $\eta$-quasi-umbilical where $\eta$ is the dual 1-form of~$JN$.
\end{lemma}

\begin{proof}
    Let $V$ be a unit vector field in $\mathcal{D}_3$.
    Then $\mathcal F = \{JN,U,JU,V,JV\}$ is a orthonormal frame on $M$, where $U=\tfrac{1}{\sin\theta}(JJ_1N-\cos\theta N)$ and $\cos \theta=g(JJ_1N,N)$.
    Note that $\theta$ cannot be zero or $\pi$, as otherwise, by the Cauchy-Schwartz inequality, $N$ lies in only one distribution.
    Moreover, it is easy to see how $J$, $J_1$, $J_2$ and $J_3$ behave with respect to~$\mathcal F$.

    We proceed by computing Equation~\eqref{lucequation}, using the curvature~\eqref{eq:flag-curvaturetensor} and plugging in the vector fields $X$, $Y$, $Z$ and $W$ as indicated in Table~\ref{tableFC3D1D2} (see Remark~\ref{remark-hor-lines}).
    The result of the computation is listed in the right column.
    From these computations it follows that the shape operator is diagonal with respect to $\mathcal F$.
    Moreover, from the first three equations we deduce that $h_{22}=h_{33}=h_{44}=h_{55}$.
\end{proof}

\begin{lemma}
    Let $M$ be a hypersurface of $F(\C^3)$ with constant sectional curvature. 
    Let $N$ be its unit normal and suppose that  $N$ has non-zero components in $\mathcal{D}_1$, $\mathcal{D}_2$ and $\mathcal{D}_3$.
    Then $M$ is totally umbilical.
\end{lemma}

\begin{proof}
    Suppose that $N=\cos\theta_1V_1+\sin\theta_1(\sin\theta_2V_2+\cos\theta_2V_3)$ where $V_1$, $V_2$ and $V_3$ are unit vectors fields in $\mathcal{D}_1$, $\mathcal{D}_2$ and $\mathcal{D}_3$, respectively.
    Then $\mathcal F = \{JV_1,JV_2,JV_3,U_1,U_2\}$ is a orthonormal frame on $M$, where $U_1=\cos\theta_2V_2-\sin\theta_2 V_3$ and $U_3=\sin\theta _1 V_1 -\cos\theta _1 \left(\sin\theta _2V_2 +\cos\theta _2 V_3 \right)$.
    Note that $\theta_1$ and $\theta_2$ cannot be zero, $\pi/2$, $3\pi/2$, nor $\pi$, since otherwise $N$ lies in one distribution, or the sum of two distributions.
    Moreover, it is easy to see how $J$, $J_1$, $J_2$ and $J_3$ behave with respect to~$\mathcal F$.

    We proceed by computing Equation~\eqref{lucequation}, using the curvature~\eqref{eq:flag-curvaturetensor} and plugging in the vector fields $X$, $Y$, $Z$ and $W$ as indicated in Table~\ref{table1FC3D1D2D3} (see Remark~\ref{remark-hor-lines}).
    The result of the computation is listed in the right column.
    The last five equations comprise a linear system with solution $h_{23} = 0$ and $h_{11}=h_{22}=h_{33}=h_{44}=h_{55}$.
    Combined with the other equations, we find that the shape operator is diagonal with respect to $\mathcal F$ and hence, $M$ is totally umbilical.
\end{proof}

\clearpage
\begin{table}[p]
    \ra{1.2} 
    \[
    \begin{NiceArray}{llllcl}
        \toprule
        X & Y & Z & W & \text{\phantom{ab}} & \cycl{X,Y,W}\ g( \tilde{R}(X,Y)Z,SW ) = 0  \\    
        \midrule
        JN & V & G(V,N) & JV && h_{15}=0 \\
        JN & V & JG(V,N) & JV &&h_{14}=0\\
        G(V,N) & JG(V,N) &JN & V && h_{44}-h_{55}=0\\
        G(V,N) & JG(V,N) &JN & JV && h_{45}=0\\[3pt]
        \hdottedline\\[-13pt]
        G(V,N) & JG(V,N) &JV & JN && h_{12}=0\\
        JN &JG(V,N)& G(V,N) & JV &&h_{23}=0\\
        JN & V & V & G(V,N) && h_{25}+2 \theta_1h_{34}=0\\
        JN & V & V & JG(V,N) && h_{24}+2 \theta_1h_{35}=0\\
        JN & JV & V & G(V,N) && h_{35}-2 \theta_2h_{34}=0\\
        JN & G(V,N) & V & JG(V,N) && h_{13}=0\\
        V &  G(V,N) & JV & JG(V,N) && h_{22}-h_{44}=0\\[3pt]
        \hdottedline\\[-13pt]
        JN &  JG(V,N)  & G(V,N) & V && h_{22}-h_{11}(1-\theta_1^2-\theta_2^2)=0\\
        JN &  JG(V,N)  & JG(V,N) & JV && h_{33}-h_{11}(1-\theta_1^2-\theta_2^2)=0\\
        JN & G(V,N) & JG(V,N) & JG(V,N) && h_{34}(2\theta_1+1)=0\\
        JN & V & JV & JG(V,N) &&h_{34} \theta_1(2\theta_1-1)=0\\    
        \bottomrule
    \end{NiceArray}
    \]
    \caption{Equation~\eqref{lucequation} in $\Ss^3\times\Ss^3$, when $PN$ is linearly independent from $N$ and $JN$.}
    \label{tableS3S3independent}
\end{table}

\begin{table}[p]
\ra{1.2}
\[
\begin{NiceArray}{llllcl}
    \toprule
    X & Y & Z & W & \text{\phantom{ab}} & \cycl{X,Y,W}\ g( \tilde{R}(X,Y)Z,SW ) = 0  \\    
    \midrule
    V_1 & V_2 & JV_1 & JV_2 && 2h_{22}+3h_{33}-5h_{55}=0\\
    V_1 & V_2 & JV_2 & JV_1 && 3h_{22}+2h_{33}-5h_{44}=0\\
    V_1 & JV_1 & V_2 & JV_2 && 5h_{22}-3h_{44}-2h_{55}=0\\  
    JN & JV_2 & JV_1 & V_2 && h_{12}=0\\
    JN & V_1 & JV_2 & JV_1 && h_{13}=0\\
    JN & JV_2 & V_1 & V_2 && h_{14}=0\\
    JN & V_1 & V_2 & JV_1 && h_{15}=0\\
    JV_1 & V_1 & V_1 & JV_2 && 5h_{23}- h_{45}=0\\
    JV_1 & V_2 & V_2 & V_1 && 4h_{24}+h_{35}=0\\
    JV_1 & V_1 & V_1 & V_2 && h_{25}-3 h_{34}=0\\[3pt]
    \hdottedline\\[-13pt]
    V_1 & V_2 & JV_1 & JV_1 && h_{23}=h_{45}=0\\
    V_2 & JV_2 & JV_1 & JV_1 && h_{24}=h_{35}=0\\
    V_1 & JV_2 & JV_1 & JV_1 && h_{25}=h_{34}=0\\  
    \bottomrule
\end{NiceArray}
\]
\caption{Equation~\eqref{lucequation} in $\Ss^3\times\Ss^3$, when $PN$ is a linear combination of $N$ and $JN$.}
\label{tableS3S3casedependent}
\end{table}

\begin{table}[p]
\ra{1.2}
\[
\begin{NiceArray}{llllcl}
    \toprule
    X & Y & Z & W & \text{\phantom{ab}} & \cycl{X,Y,W}\ g( \tilde{R}(X,Y)Z,SW ) = 0 \\
    \midrule
    V & JV & G(V,N) & JG(V,N) && h_{22}+h_{33}-2h_{55}=0\\
    V  & G(V,N) & JV & JG(V,N) && 2h_{22}-h_{44}-h_{55}=0\\
    V & JV & JG(V,N) & G(V,N) && h_{22}+h_{33}-2h_{44}=0\\
    JN & V & JV & JV && h_{12}=0\\
    JN & V & V & JV && h_{13}=0\\
    JN & V & V & G(V,N) && h_{14}=0\\
    JN & V & V & JG(V,N) && h_{15}=0\\
    V & JG(V,N) & V & G(V,N) && h_{23}=0\\
    JV & V & G(V,N) & G(V,N) && h_{45}=0\\
    V & JV &JV& JG(V,N) && 3h_{25}-h_{34}=0\\
    V & JV &V& JG(V,N) && h_{24}+3h_{35}=0\\[3pt]
    \hdottedline\\[-13pt]
    V & JV &V& G(V,N) && h_{25}=h_{34}=0\\
    V & JG(V,N)&JG(V,N)& G(V,N) && h_{24}=h_{35}=0\\
    \bottomrule
\end{NiceArray}
\]
\caption{Equation~\eqref{lucequation} in $\C P^3$, when $N \in \mathcal D_1$.}
\label{tableCP3D1}
\end{table}

\begin{table}[p]
\ra{1.2}
\[
\begin{NiceArray}{llllcl}
    \toprule
    X & Y & Z & W & \text{\phantom{ab}} & \cycl{X,Y,W}\ g( \tilde{R}(X,Y)Z,SW ) = 0 \\
    \midrule
    V & JV & G(V,N) & JG(V,N) && h_{22}+h_{33}-2h_{55}=0\\
    V  & G(V,N) & JV & JG(V,N) && 2h_{22}-h_{44}-h_{55}=0\\
    V & JV & JG(V,N) & G(V,N) && h_{22}+h_{33}-2h_{44}=0\\
    JN & V & JV & JV && h_{12}=0\\
    JN & V & V & JV && h_{13}=0\\
    JN & V & JV & JG(V,N) && h_{14}=0\\
    JN & V & G(V,N) & JV && h_{15}=0\\
    V & JG(V,N) & V & G(V,N) && h_{23}=0\\
    JV & V & G(V,N) & G(V,N) && h_{45}=0\\
    V & JV &JV& JG(V,N) && 7h_{25}+3h_{34}=0\\
    V & JV &V& JG(V,N) && 3h_{24}-7h_{35}=0\\[3pt]
    \hdottedline\\[-13pt]
    V & JV &V& G(V,N) && h_{25}=h_{34}=0\\
    V & JG(V,N)&JG(V,N)& G(V,N) && h_{24}=h_{35}=0\\
    \bottomrule
\end{NiceArray}\]
\caption{Equation~\eqref{lucequation} in $\C P^3$, when $N \in \mathcal D_2$.}
\label{tableCP3D2}
\end{table}

\begin{table}[p]
\ra{1.2}
\[
\begin{NiceArray}{llllcl}
    \toprule
     X & Y & Z & W & \text{\phantom{ab}} & \cycl{X,Y,W}\ g( \tilde{R}(X,Y)Z,SW ) = 0 \\
    \midrule
    JV_1 & JV_2 & G(V_1,V_2) & G(V_1,V_2) && h_{12}=0\\
    JV_1 &  G(V_1,V_2) & JV_1 & JG(V_1,V_2) && h_{13}=0\\
    JV_1 & JV_2 & JV_2 & G(V_1,V_2) && h_{14}=0\\
    JV_1 & JV_2 & JV_2 & JG(V_1,V_2) && h_{15}=0\\[3pt]
    \hdottedline\\[-13pt]
    JV_1 & JV_2 & JV_1 & V_3 && h_{23}=0\\
    JV_1 & JV_2 & V_3 & JG(V_1,V_2) && h_{24}=0\\
    JV_1 & JV_2 & V_3 & G(V_1,V_2) && h_{25}=0\\
    JV_1 & G(V_1,V_2) & JV_1 & V_3 && h_{34}=0\\
    JV_1 & G(V_1,V_2) & G(V_1,V_2) & V_3 && h_{45}=0\\
    JV_1 & JV_2 & V_3 & V_3 && h_{11}\sin^2\theta-h_{22}\cos^2\theta=0\\[3pt]
    \hdottedline\\[-13pt]
    JV_1 & G(V_1,V_2) & V_3 & V_3 && h_{35}=0\\
    V_3 & JG(V_1,V_2) & JV_1 & G(V_1,V_2) && -2h_{33}+h_{44}+h_{55}=0\\
    JV_2 & G(V_1,V_2) & JG(V_1,V_2) & V_3 && h_{22}\cos^2\theta +h_{33}-2h_{44}=0\\
    JV_2 & V_3 & G(V_1,V_2) & JG(V_1,V_2) && h_{22} \cos^2\theta+h_{33}-2h_{55}=0\\
    \bottomrule
\end{NiceArray}
\]
\caption{Equation~\eqref{lucequation} in $\C P^3$, when $N = V_1+V_2$ where $V_i \in \mathcal D_i \smallsetminus\{0\}$.}
\label{tableCP3D1D2}
\end{table}

\begin{table}[p]
\ra{1.2}
\[
\begin{NiceArray}{llllcl}
    \toprule
    X & Y & Z & W & \text{\phantom{ab}} & \cycl{X,Y,W}\ g( \tilde{R}(X,Y)Z,SW ) = 0\\
    \midrule
    U & JU & V & JV && h_{22}+h_{33}-2h_{55}=0\\
    U & JU & JV & V && h_{22}+h_{33}-2h_{44}=0\\
    U & V & JU & JV && 2h_{22}-h_{44}-h_{55}=0\\
    JN & U & JU & JU && h_{12}=0 \\
    JN & U & U & JU && h_{13}=0\\
    JN & U & JU & JV && h_{14}=0\\
    JN & U & JU & V && h_{15}=0\\
    U & V & U & JV && h_{23}=0\\
    U & JU & V & V && h_{45}=0\\
    U & JU & U & V && h_{25}+5h_{34}=0\\
    U & JU & U & JV && h_{24}-5h_{35}=0\\[3pt]
    \hdottedline\\[-13pt]
    U & JU & JU & JV && h_{25}=h_{34}=0\\
    U & JU & JU & V && h_{24}=h_{35}=0\\
    \bottomrule
\end{NiceArray}
\]
\caption{Equation~\eqref{lucequation} in $F(\C^3)$, when $N \in \mathcal D_1$.}
\label{tableFC3D1}
\end{table}

\begin{table}[p]
    \ra{1.2}
    \[
    \begin{NiceArray}{llllcl}
        \toprule
        X & Y & Z & W & \text{\phantom{ab}} & \cycl{X,Y,W}\ g( \tilde{R}(X,Y)Z,SW ) = 0 \\
        \midrule
        U & JU & V & JV && h_{22}+h_{33}-2h_{55}=0\\
        U & JU & JV & V && h_{22}+h_{33}-2h_{44}=0\\
        U & JV & JU & V && 2h_{22}-h_{44}-h_{55}=0\\
        JN & JU & V & JV && h_{12}=0 \\
        JN & U & V & JV && h_{13}=0\\
        JN & JU & U & JV && h_{14}=0\\
        JN & U & V & JU && h_{15}=0\\
        U & V & U & JV && h_{23}=0\\
        JN & U & JN & V && h_{24}=0\\
        JN & U & JN & JV && h_{25}=0\\
        U & JV & JV & JU && h_{45}=0\\[3pt]
    \hdottedline\\[-13pt]
        U & JU & JU & JV && h_{34}=0\\
        U & JU & JU & V && h_{35}=0\\
        \bottomrule
    \end{NiceArray}
    \]
    \caption{Equation~\eqref{lucequation} in $F(\C^3)$, when $N \in \mathcal D_1 \oplus \mathcal D_2$.}
    \label{tableFC3D1D2}
    \end{table}

\begin{table}[p]
    \ra{1.2}
    \[
    \begin{NiceArray}{llllcl}
        \toprule
        X & Y & Z & W & \text{\phantom{ab}} & \cycl{X,Y,W}\ g( \tilde{R}(X,Y)Z,SW ) = 0\\
        \midrule
        JV_1 & JV_2 & JV_1 & U_1 && h_{15}=0 \\[3pt]
    \hdottedline\\[-13pt]
        JV_1 & JV_2 & JV_1 & U_2 && h_{14}\cos\theta_2 + 5 h_{25} \sin\theta_1=0 \\[3pt]
    \hdottedline\\[-13pt]
        JV_1 & JV_2 & JV_3 & U_1&& h_{14}=h_{25}=0\\[3pt]
    \hdottedline\\[-13pt]
        JV_1 & JV_2 & U_1 & JV_3 && h_{35}=0\\
        JV_1 & JV_2 & JV_2 & U_2 && h_{24}=0\\[3pt]
    \hdottedline\\[-13pt]
        JV_1 & U_2 & U_1 & U_1 && h_{34}=0\\
        JV_1 & U_1 & U_1 & JV_3 &&3 h_{12} \cos \theta_2+5 h_{13} \sin \theta_2 +h_{45}\sin\theta_1 = 0\\
        JV_1 & JV_2 & U_1 & U_1 && 5 h_{12} \cos\theta_2 + h_{13} \sin\theta_2 - h_{45} \sin\theta_1 =0\\[3pt]
    \hdottedline\\[-13pt]
        JV_1& U_1 & JV_1 & U_2 && h_{12}=h_{13}=h_{45}=0\\[3pt]
    \hdottedline\\[-13pt]
        JV_1& JV_2 & U_1 & U_2 && h_{11}- 2h_{22}+ h_{55}+ 2 h_{23} \tan\theta_2 = 0 \\
        JV_1& JV_2 & U_2 & U_1 && -2h_{11} + h_{22} + h_{44} - h_{23} \tan\theta_2 = 0 \\
        JV_1& U_1 & JV_2 & U_2 && h_{11}-2h_{44}+h_{55} = 0\\
        JV_2 & U_2 & JV_2 & U_1 && 5 h_{44}  - 5 h_{55} + 2 h_{23} \csc 2\theta_2 = 0 \\
        JV_1 & JV_3 & U_1 & U_2 && h_{11} - 2 h_{33}  + h_{55}  + 2h_{23}\cot\theta_2 = 0 \\
        \bottomrule
    \end{NiceArray}
    \]
    \caption{Equation~\eqref{lucequation} in $F(\C^3)$, when $N \in \mathcal D_1 \oplus D_2 \oplus D_3$.}
    \label{table1FC3D1D2D3}
    \end{table}

\clearpage

\bibliographystyle{abbrv}
\bibliography{allpapers}

\end{document}